\documentclass[11pt]{article}
\usepackage{a4}
\usepackage{tikz}
\usepackage{latexsym}
\usepackage{amsfonts}
\usepackage{amssymb}
\usepackage{amsmath}
\usepackage{amsthm}
\usepackage{float}
\newtheorem{theorem}{Theorem}[section]
\newtheorem{lemma}[theorem]{Lemma}

\title{Sum-ratio estimates over arbitrary finite fields}
\author{Oliver Roche-Newton\footnote{ Part of this research was performed while the author was visiting the Institute for Pure and Applied Mathematics (IPAM), which is supported by the National Science Foundation. The author was also supported by EPSRC Doctoral Prize Scheme (Grant Ref:  EP/K503125/1) and by the Austrian Science Fund (FWF): Project F5511-N26, which is part of the Special Research Program ``Quasi-Monte Carlo Methods: Theory and Applications.}}

\begin{document}
\maketitle

\begin{abstract} The aim of this note is to record a proof that the estimate
$$\max{\{|A+A|,|A:A|\}}\gg{|A|^{12/11}}$$
holds for any set $A\subset{\mathbb{F}_q}$, provided that $A$ satisfies certain conditions which state that it is not too close to being a subfield. An analogous result was established in \cite{LiORN}, with the product set $A\cdot{A}$ in the place of the ratio set $A:A$. The sum-ratio estimate here beats the sum-product estimate in \cite{LiORN} by a logarithmic factor, with slightly improved conditions for the set $A$, and the proof is arguably a little more intuitive. The sum-ratio estimate was mentioned in \cite{LiORN}, but a proof was not given.

\end{abstract}

\section{Introduction}

Given a set $A\subset{\mathbb{F}_q^*}$, define the \textit{sum set} by $A+A=\{a+b:a,b\in{A}\}$ and the \textit{product set} by $A\cdot{A}=\{ab:a,b\in{A}\}$. It is expected that at least one of these sets will be in some sense ``large", provided that we are not in a degenerate case in which $A$ is a subfield. One way to avoid these degenerate cases is to assume that $q$ is prime and $|A|\leq{\sqrt{q}}$, and in recent years, there have been a succession of papers which have given improved quantitative bounds for sum-product estimates in this range. At the time of writing, the best known estimate is due to Rudnev \cite{mishaSP}, who proved, under the aforementioned conditions, that

$$\max{\{|A+A|,|A\cdot{A}|\}}\gg{\frac{|A|^{12/11}}{(\log{|A|})^{4/11}}}.$$

The following result from \cite{LiORN} generalised Rudnev's sum-product estimate to the case whereby $q$ is not prime:

\begin{theorem}\label{theorem:main2}
Let $A$ be a subset of $\mathbb{F}_{q}^*$. If

$$|A\cap{cG}|<{|G|^{1/2}}$$ 

for  any subfield $G$ of
$\mathbb{F}_{q}$ and any element $c\in{\mathbb{F}_q}$, then
\[\max\{|A+A|,|A\cdot{A}|\}\gg{\frac{|A|^{12/11}}{(\log|A|)^{5/11}}}.\]

\end{theorem}

Define the \textit{ratio set} by $A:A=\{a/b:a,b\in{A}\}$. One expects that similar results can be attained if the ratio set replaces the product set in Theorem \ref{theorem:main2}. The main result of this note proves that this is indeed the case:

\begin{theorem}\label{theorem:ratios} 

Suppose that $A$ is a subset of $\mathbb{F}_q^*$ with the property that 

$$|A\cap{cG}|\leq{\max\left\{|G|^{1/2},\frac{|A|}{8}\right\}}$$ 

for any subfield $G$ of $\mathbb{F}_q$ and any element $c\in{\mathbb{F}_q}$. Then either

$$|A+A|^7|A:A|^4\gg{|A|^{12}},$$
or
$$|A+A|^6|A:A|^5\gg{|A|^{12}}.$$
In particular, it follows that
$$\max\{|A+A|,|A:A|\}\gg{|A|^{12/11}}.$$

\end{theorem}

Note that, in comparison with Theorem \ref{theorem:main2}, the subfield intersection condition is loosened slightly in this statement, with the additional information that the sum-ratio estimate holds if $|A\cap{cG}|\leq{\frac{|A|}{8}}$. This was due to an oversight in \cite{LiORN}, and the statement of Theorem \ref{theorem:main2} can be strengthened similarly by lengthening the proof slightly\footnote{To be more specific, the additional case 5.2 can be found in the proof of Theorem \ref{theorem:ratios}, but not in the proof of Theorem \ref{theorem:main2}. By adding in this extra case to the proof of the sum-product estimate, one obtains a more general result.}.

We remark that the sum-ratio estimate was alluded to in both \cite{mishaSP} and \cite{LiORN}, although a proof was not given in either case. Although the proof of Theorem \ref{theorem:ratios} is structurally similar to that of Theorem \ref{theorem:main2}, it is subtly different in a number of places, and perhaps not entirely obvious. The motivation for carefully recording the result comes from an intended application for polynomial orbits and sum-product type estimates involving polynomials in \cite{ShORN}, for which a proof should be provided for completeness. Furthermore, the more straightforward nature of the pigeonholing in this proof makes it more accessible, which could potentially be helpful for future research in this direction.

Observe that there is no logarithmic factor in the statement of Theorem \ref{theorem:ratios}. The absolute constant hidden in the $\gg$ symbol can be kept track of, although it isn't here.

\subsubsection*{Notation}
We recall that the
notations $U \ll V$ and  $V \gg U$ are both equivalent to the
statement that the inequality $|U| \le c V$ holds with some
constant $c> 0$. If $U\ll{V}$ and $U\gg{V}$, then we may write $U\approx{V}$.

The multiplicity of an element $x$ of the ratio set is written as $r_{A:A}(x)$, so that $r_{A:A}(x)=|\{(a,b)\in{A\times{A}}:a/b=x\}|$.

\section{Preliminary results}

A few preliminary results will be called upon from other papers. The first of these has been extracted from case 2 in the proof of the
main theorem in Rudnev \cite{mishaSP}. A proof of this statement can also be found in \cite{LiORN} (see Lemma 2.4).

First recall from previous finite field sum-product estimates the definition of $R(B)$, for any $B\subset{\mathbb{F}_q}$, to be the
set
\begin{equation}
\label{Rdefn}
R(B):=\left\{\frac{b_1-b_2}{b_3-b_4}:b_1,b_2,b_3,b_4\in{B},b_3\neq{b_4}\right\}.
\end{equation}

\begin{lemma}\label{lemma 21}
  Let $B\subset{\mathbb{F}_q}$ with $|R(B)|\gg|B|^2$.
  Then there exist elements $a,b,c,d\in B$ such that for any subset $B'\subset B$
  with $|B'|\approx|B|$,
 $$|(a-b)\cdot B'+(c-d)\cdot B'|\gg|B|^2.$$

\end{lemma}

The next result which will be needed is the Pl\"{u}nnecke-Ruzsa inequality:

\begin{lemma}\label{theorem:SPlun}Let $X,B_1,...,B_k$ be subsets of a field $F$. Then
$$|B_1+\cdots+B_k|\leq{\frac{|X+B_1|\cdots|X+B_k|}{|X|^{k-1}}}.$$

\end{lemma}

By applying Lemma \ref{theorem:SPlun} iteratively, the following corollary was established by Katz and Shen \cite{KS}.

\begin{lemma}\label{theorem:CPlun}
Let $X,B_1,...,B_k$ be subsets of a field $F$. Then for any $\epsilon\in{(0,1)}$,
there exists a subset $X'\subseteq{X}$, with
$|X'|\geq{(1-\epsilon)|X|}$, and some constant $C(\epsilon)$, such
that
$$|X'+B_1+\cdots+B_k|\leq{C(\epsilon)\frac{|X+B_1|\cdots|X+B_k|}{|X|^{k-1}}}.$$

\end{lemma}

We will need the following covering lemma, which appeared in sum-product estimates for the first time in Shen \cite{ShenSP}.

\begin{lemma}\label{theorem:covering}
Let $X$ and $Y$ be additive sets. Then for any $\epsilon\in{(0,1)}$ there is some constant $C(\epsilon)$, such that at least $(1-\epsilon)|X|$ of the elements of $X$ can be covered by $C(\epsilon)\frac{\min\{|X+Y|,|X-Y|\}}{|Y|}$ translates of $Y$.

\end{lemma}

The main new tool that appeared in \cite{LiORN} was the following result:

\begin{lemma}\label{theorem:polynomial}
Let $B$ be a subset of $\mathbb{F}_q$ with at least two elements, and let
$\mathbb{F}_{B}$ denote the subfield generated by $B$. Then there exists
a polynomial of several variables with integer coefficients
$P(x_1,x_2,\ldots,x_{m})$ such that
$P(B,B,\ldots,B)=\mathbb{F}_{B}.$
\end{lemma}

\section{Proof of Theorem \ref{theorem:ratios}}

At the outset, apply Lemma \ref{theorem:CPlun} to identify some subset $A'\subset{A}$, with cardinality $|A'|\approx{|A|}$, so that
\begin{equation}
|A'+A'+A'+A'|\ll{\frac{|A+A|^3}{|A|^2}}.
\label{lastbit}
\end{equation}

Since many more refinements of $A$ are needed throughout the proof, this first change is made without a change in notation. So, throughout the rest of the proof, when the set $A$ is referred to, we are really talking about the large subset $A'$. In the conclusion of three of the five cases that follow, the following inequality will be applied:

\begin{equation}
|A+A+A+A|\ll{\frac{|A+A|^3}{|A|^2}}.
\label{finally}
\end{equation}

Consider the point set $A\times{A}\subset{\mathbb{F}_q\times{\mathbb{F}_q}}$. The line through the origin with gradient $\xi$ is the set $\{(x,y)\in{\mathbb{F}_q}:y=\xi x\}$. Label this line $L_{\xi}$, and observe that

$$\sum_{\xi\in{A:A}}|L_{\xi}\cap{(A \times A)}|=\sum_{\xi\in{A:A}}r_{A:A}(\xi)=|A|^2.$$

By the pigeonhole principle, a positive proportion of points in this set are supported on popular lines through the origin - that is lines that contain, say, more than half the average number of points from $A\times{A}$. To be precise, define the set of ``rich" slopes to be the set

$$\Xi_{rich}:=\left\{\xi\in{A:A}:r_{A:A}(\xi)\geq{\frac{|A|^2}{2|A:A|}}\right\}.$$

$$|A|^2=\sum_{\xi\in{\Xi_{rich}}}r_{A:A}(\xi)+\sum_{\xi\notin{\Xi_{rich}}}r_{A:A}(\xi).$$

Since the second term on the RHS contributes at most $\frac{|A|^2}{2}$, it follows that

$$\sum_{\xi\in{\Xi_{rich}}}r_{A:A}(\xi)\geq{\frac{|A|^2}{2}}.$$

We define $P$ to be the set of all points from $A\times{A}$ lying on a line through the origin supporting at least $\frac{|A|^2}{2|A:A|}$ points from $A\times{A}$, i.e. the points with a ``rich" slope. We have established that $|P|\geq{\frac{|A|^2}{2}}$. By elementary pigeonholing, there exists some popular abscissa $x_*$, so that the set

$$A_{x_*}=\{y:(x_*,y)\in{P}\},$$ 

has cardinality $|A_{x_*}|\geq{\frac{|A|}{2}}$.

For any point $p=(x_0,y_0)\in{P}$, let $P_{y_0/{x_0}}$ be the projection of points in $P$ on the line through the origin supporting $(x_0,y_0)$, onto the $x$-axis. So,

$$P_{y_0/{x_0}}=\left\{x:\left(x,\frac{xy_0}{x_0}\right)\in{P}\right\}.$$

In particular, note that for all $y\in{A_{x_*}}$, $|P_{y/x_*}|\gg{\frac{|A|^2}{|A:A|}}$. Another important property is the fact that $\frac{y}{x_*}P_{y/x_*} \subseteq A$.

Next this process may be repeated. Consider the point set $A_{x_*}\times{A_{x_*}}$, which has cardinality at least $\frac{|A|^2}{4}$. Once again, we may refine this point set by deleting points on unpopular lines. To be precise, let $S$ denote the set
$$S:=\left\{(x,y)\in{A_{x_*}\times{A_{x_*}}}:r_{A_{x_*}:A_{x_*}}(y/x)\geq{\frac{|A|^2}{8|A:A|}}\right\}.$$
 
By the same argument that established that $|P|\geq{\frac{|A|^2}{2}}$, it follows that the point set $S$ has cardinality $|S|\geq{\frac{|A|^2}{8}}$. Again there is a popular abscissa, $x_0$, so that the set 

$$A_{x_0}=\{y:(x_0,y)\in{S}\},$$ 

has cardinality $|A_{x_0}|\gg{|A|}$. Since the sum-ratio problem, and the conditions of Theorem \ref{theorem:ratios} are invariant under dilation, we may assume without loss of generality that $x_0=1$.

For some element $y\in{A_{1}}=A_{x_0}$, we will be interested in the projection of points in $S$ on the line connecting the origin and the point $(1,y)$, down onto the $x$-axis. This set of values can be defined more precisely as the set

\begin{equation}
S_{y}=\{x:(x,xy)\in{S}\}.
\label{richslopedefn}
\end{equation}

Note, for any $y\in{A_{1}}$, that $S_{y},yS_y\subseteq{A_{x_*}}$, and that crucially,

\begin{equation}
|S_{y}|\gg{\frac{|A|^2}{|A:A|}}.
\label{richslopes}
\end{equation}

\subsection{Five Cases}

The proof is now divided into five cases corresponding to the nature of the set $R(A_1)$.

\textbf{Case 1} - $R(A_1)\neq{R(A_{x_*})}$:

Since $A_1\subseteq{A_{x_*}}$, it must be the case that $R(A_1)\subseteq{R(A_{x_*})}$. Therefore, the only possibility for this case is that this inclusion is proper. So, there must be some element $r\in{R(A_{x_*})}$ such that $r\notin{R(A_1)}$. Fix this $r=\frac{a-b}{c-d}$ and elements $a,b,c,d\in{A_{x_*}}$ representing it. Since $r\notin{R(A_1)}$, for any subset $A_1'$ of $A_1$, there exist only trivial solutions to

\begin{equation}
b_1+rb_2=b_3+rb_4,
\label{cseq2}
\end{equation}

such that $b_1,b_2,b_3,b_4\in{A_1'}$. The absence of non-trivial solutions to \eqref{cseq2} implies that 

$$|A_1'|^2=|A_1'+rA_1'|.$$

After expanding out the above expression and dilating the long sum set, it follows that

$$|A_1'|^2\ll{\left|\frac{c}{x_*}A_1'-\frac{d}{x_*}A_1'+\frac{a}{x_*}A_1'-\frac{b}{x_*}A_1'\right|}.$$

At least 90\% of $\frac{c}{x_*}A_1$ can be covered by at most

$$\frac{\left|\frac{c}{x_*}A_1+\frac{c}{x_*}P_{c/x_*}\right|}{|P_{c/x_*}|}\ll{\frac{|A+A||A:A|}{|A|^2}}$$

translates of $\frac{c}{x_*}P_{c/x_*}\subset{A}$. Similarly, each of $-\frac{d}{x_*}A_1$, $\frac{a}{x_*}A_1$ and $-\frac{b}{x_*}A_1$ can be 90\% covered by $\ll{\frac{|A+A||A:A|}{|A|^2}}$ translates of $A$.

By choosing an appropriate subset $A_1'$ of size $|A_1'|\approx{|A_1|}$, we can ensure that each of $\frac{c}{x_*}A_1'$, $-\frac{d}{x_*}A_1'$, $\frac{a}{x_*}A_1'$ and $-\frac{b}{x_*}A_1'$ get fully covered by these translates of $A$. Therefore, the covering lemma is applied four times in order to deduce that

\begin{equation}
|A|^2\ll{\frac{|A+A+A+A||A+A|^4|A:A|^4}{|A|^8}}.
\label{case1.1}
\end{equation}

After applying \eqref{lastbit}, it follows that

\begin{equation}
|A|^2\ll{\frac{|A+A|^7|A:A|^4}{|A|^{10}}},
\label{case1.2}
\end{equation}

as required.

From this point forward, we may assume that $R(A_1)=R(A_{x_*})$.

\textbf{Case 2} - $1+R(A_1)\nsubseteq{R(A_1)}$:

In this case, there exist elements $a,b,c,d\in{A_1}$ such that

$$r=1+\frac{a-b}{c-d}\notin{R(A_1)}=R(A_{x_*}).$$

Now, recall the set $S_a$ defined earlier. Let $S_a'$ be a subset of $S_a$ such that $|S_a'|\approx{|S_a|}$, and similarly let $A_1'$ be a positively proportioned subset of $A_1$. These two subsets will be specified later in order to apply the covering lemma effectively.

By Lemma \ref{theorem:CPlun} with $X=(c-d)A_1'$, a further subset $A_1''\subseteq{A_1'}$, with $|A_1''|\approx{|A_1'|}\approx{|A|}$, can be identified such that:

\begin{align}
|A_1''+rS_a'|&\leq{|(c-d)A_1''+(c-d)S_a'+(a-b)S_a'|}
\\&\ll{\frac{|A_1'+S_{a}'|}{|A_1'|}|(c-d)A_1'+(a-b)S_{a}'|}.
\label{plunnecker2}
\end{align}

Since $A_1''$ and $S_a'$ are subsets of $A_{x_*}$, there exist only trivial solutions to

$$a_1+ra_2=a_3+ra_4,$$

such that $a_1,a_3\in{A_1''}$ and $a_2,a_4\in{S_a'}$, otherwise $r\in{R(A_{x_*})}$, which is a contradiction. Therefore,

$$\frac{|A|^3}{|A:A|}\approx{|A_1''||S_a'|}=|A_1''+rS_a'|.$$

Combining this knowledge with \eqref{plunnecker2}, it follows that

\begin{equation}
\frac{|A|^4}{|A:A|}\ll{|A+A||cA_1'-dA_1'+aS_a'-bS_a'|}.
\label{messy3}
\end{equation}

At least 90\% of $cA_1$ can be covered by at most,

$$\frac{|cA_1+cS_c|}{|cS_c|}\ll{\frac{|A+A||A:A|}{|A|^2}}$$

translates of $cS_c\subset{A}$. Similarly, $-dA_1$ can be 90\% covered by $\ll{\frac{|A+A||A:A|}{|A|^2}}$ translates of $A$. The subset $A_1'$ can be chosen earlier in the proof in such a way as to ensure that both $cA_1'$ and $-dA_1'$ get fully covered by these translates of $A$. In much the same way, 90\% of $-bS_a$ can be covered by at most

$$K\frac{|-bS_a-bS_b|}{|bS_b|}\ll{\frac{|A+A||A:A|}{|A|^2}}$$

translates of $bS_b\subset{A}$. The subset $S_a'$ can be chosen earlier in the proof so that $-bS_a'$ gets fully covered by these translates of $A$. Working from \eqref{messy3} and applying the covering lemma three times, it follows that

$$|A|^{10}\ll{|A:A|^4|A+A|^4|A+A+aS_a'+A|}.$$

Finally, observe that $aS_a'$ is a subset of $A$, and thus there is no need to apply the covering lemma for this term. This gives

\begin{equation}
|A|^{10}\ll{|A:A|^4|A+A|^4|A+A+A+A|},
\label{case2.1}
\end{equation}

and finally, applying \eqref{finally}, we conclude that

\begin{equation}
|A+A|^7|A:A|^4\gg{|A|^{12}}.
\label{case2.2}
\end{equation}

\textbf{Case 3} - $A_1\nsubseteq{R(A_1)}$:

In this case, there exists some $a\in{A_1}$ such that $a\notin{R(A_1)}$. Then for any subset $A_1'$ of $A_1$, it follows that

$$|A_1'|^2={|A_1'+aA_1'|}.$$

By Lemma \ref{theorem:covering}, at least 90\% of $aA_1$ can be covered by at most

$$\frac{|aA_1+aS_a|}{|aS_a|}\ll{\frac{|A+A||A:A|}{|A|^2}}$$

translates of $aS_a\subset{A}$. $A_1'$ can then be chosen so that $|A_1'|\approx{|A|}$ and $aA_1'$ is covered entirely by these translates. Therefore, 

\begin{equation}|A|^4\ll{|A+A|^2|A:A|},
\label{case3}
\end{equation}

a result which is considerably stronger than the one we are seeking to prove.

\textbf{Case 4} - $A_1R(A_1)\nsubseteq{R(A_1)}:$

In this case, there must exist some $a,c,d,e,f\in{A_1}$ such that

$$r=a\frac{c-d}{e-f}\notin{R(A_1)}=R(A_{x_*}).$$

Let $Y_1$ be a subset of $A_{x_*}$, to be chosen later. Recall also that $S_a$ is a subset of $A_{x_*}$. Since $r\notin{R(A_{x_*})}$ there exist only trivial solution to

$$a_1+ra_2=a_3+ra_4,$$

such that $a_1,a_3\in{Y_1}$ and $a_2,a_4\in{S_a}$. Therefore,

$$|Y_1||S_a|=|Y_1+rS_a|.$$

Next apply Lemma \ref{theorem:SPlun} with $X=\frac{c-d}{e-f}Y_2$; the set $Y_2$ will be specified later. We obtain

\begin{align*}
|Y_2||Y_1||S_a|&=|Y_2||Y_1+rS_a|
\\&\leq{\left|Y_1+\frac{c-d}{e-f}Y_2\right||aS_a+Y_2|}
\\&\leq{\left|eY_1-fY_1+cY_2-dY_2\right||aS_a+Y_2|}.
\end{align*}

The sets $Y_1$ and $Y_2$ may be chosen to be subsets of $S_e$ and $S_c$ respectively. Then, since $eS_e,aS_a,cS_c\subset{A}$, it follows that

$$|Y_1||Y_2||S_a|\leq{|A-fY_1+A-dY_2||A+A|}.$$

Next, we need to apply the covering lemma twice. At least 90\% of $-fS_e$ can be covered by at most

$$\frac{|-fS_e-fS_f|}{|fS_f|}\ll{\frac{|A+A||A:A|}{|A|^2}}$$

translates of $fS_f\subseteq{A}$. The set $Y_1$ may be chosen so that $|Y_1|\approx{|S_e|}$ and $-fY_1$ is covered completely by these translates of $A$. In much the same way, $Y_2$ can be chosen so that $|Y_2|\approx{|S_c|}$ and $-dY_2$ is covered by $\ll{\frac{|A+A||A:A|}{|A|^2}}$ translates of $A$. It follows that

\begin{equation}
\frac{|A|^6}{|A:A|^3}\ll{\frac{|A+A+A+A||A+A|^3|A:A|^2}{|A|^4}}
\end{equation}

Rearranging this inequality yields

\begin{equation}
|A:A|^5|A+A|^3|A+A+A+A|\gg{|A|^{10}}.
\label{case4.1}
\end{equation}

Finally, applying \eqref{finally}, it follows that

\begin{equation}
|A:A|^5|A+A|^6\gg{|A|^{12}}.
\label{case4.2}
\end{equation}

\textbf{Case 5} 
Suppose Cases $1\sim 4$ don't happen. Then in
particular we have
\begin{align}\label{338}
A_1&\subseteq R(A_1);\\
1+R(A_1)&\subseteq R(A_1);\\
A_1R(A_1)&\subseteq R(A_1).
\end{align}
Since $|A_1R(A_1)|\geq|R(A_1)|$,
\begin{align*}
A_1R(A_1)=R(A_1).
\end{align*}
Noting that $R(A_1) \setminus \{0\}$ is closed under reciprocation, it follows that
\begin{align*}
\frac{R(A_1)}{A_1}=R(A_1).
\end{align*}
Given $a,x,y,z,w\in A_1$ with $z\neq w$,
\[a+\frac{x-y}{z-w}=a\cdot\left(1+\frac{1}{a}\cdot\frac{x-y}{z-w}\right)\in R(A_1).\]
This implies that
\begin{align*}
A_1+R(A_1)=R(A_1).
\end{align*}
Noting that $R(A_1)$ is additively symmetric
(that is, $R(A_1)=-R(A_1)$), we have
\begin{align*}
R(A_1)-A_1=R(A_1).
\end{align*}
We also note that
 \begin{align*}
A_1A_1+R(A_1)\subseteq A_1\left(A_1+\frac{R(A_1)}{A_1}\right)&=A_1(A_1+R(A_1))
\\&=A_1R(A_1)
\\&=R(A_1).
 \end{align*}
By induction, it is easy to show that
 \begin{align*}
 A_1^{(n)}+R(A_1)=R(A_1),
 \end{align*}
where $A_1^{(n)}$ is the $n$-fold product set of $A_1$. Consequently, for any polynomial of several variables with integer coefficients $P(x_1,x_2,\ldots,x_m)$,
 \begin{align*}
 P(A_1,A_1,\ldots,A_1)+R(A_1)=R(A_1).
 \end{align*}
 Applying Lemma \ref{theorem:polynomial},  we have
$\mathbb{F}_{A_1}+R(A_1)=R(A_1),$
where $\mathbb{F}_{A_1}$ is the subfield generated by
$A_1$. Since
\begin{align*}\mathbb{F}_{A_1}\subseteq
\mathbb{F}_{A_1}+R(A_1)=R(A_1)\subseteq \mathbb{F}_{A_1},
\end{align*}
we get $$R(A_1)=\mathbb{F}_{A_1}.$$ 

Hence, according to the conditions of Theorem \ref{theorem:ratios}, there are two possible cases.

Case 5.1: $|A\cap{R(A_1)}|<{|R(A_1)|^{1/2}}$.

Then,

\begin{align*}
|R(A_1)|&>{|R(A_1)\cap{A}|^2}
\\&\geq{|R(A_1)\cap{A_1}|^2}
\\&=|A_1|^2,
\end{align*}

where the latter equality is a consequence of the fact that we are not in case 3. By Lemma \ref{lemma 21}, there exist four elements $a,b,c,d\in{A_1}$, such that for any $A_1'\subset{A_1}$ with $|A_1'|\approx{|A_1|}$, 

$$|aA_1'-bA_1'+cA_1'-dA_1'|\gg{|A|^2}.$$

Applying the covering lemma, we see that 90\% of $aA_1$ can be covered by at most

$$\frac{|aA_1+aS_a|}{|aS_a|}\ll{\frac{|A+A||A:A|}{|A|^2}}$$

translates of $aS_a\subset{A}$. In much the same way, each of $-bA_1$, $cA_1$ and $-dA_1$ can be 90\% covered by $\ll{\frac{|A+A||A:A|}{|A|^2}}$ translates of $A$. The set $A_1'$ can be chosen so that $aA_1'$, $-bA_1'$, $cA_1'$ and $-dA_1'$ are fully covered by the translates of $A$. After applying the covering lemma four times, it follows that
\begin{equation}
|A|^2\ll{|A+A+A+A|\frac{|A+A|^4|A:A|^4}{|A|^8}}.
\label{case5.1}
\end{equation}

Applying \eqref{finally}, it follows that

\begin{equation}
|A+A|^7|A:A|^4\gg{|A|^{12}}
\label{case5.2}
\end{equation}

Case 5.2: $|A\cap{R(A_1)}|<{|A|/8}$.

Then,

\begin{align*}
\frac{|A|}{8}&>{|A\cap{R(A_1)}|}
\\&\geq{|A_1\cap{R(A_1)}|}
\\&={|A_1|}
\geq{\frac{|A|}{8}}.
\end{align*}

We obtain a contradiction here, and so this case cannot occur.

\begin{flushright}
\qedsymbol
\end{flushright}

\section{Estimates for iterated sum sets}

We conclude by pointing out that one can obtain slightly better exponents by considering longer sum sets:

\begin{theorem}\label{theorem:ratios2} 

Suppose that $A$ is a subset of $\mathbb{F}_q^*$ with the property that 

$$|A\cap{cG}|\leq{\max\left\{|G|^{1/2},\frac{|A|}{8}\right\}}$$ 

for any subfield $G$ of $\mathbb{F}_q$ and any element $c\in{\mathbb{F}_q}$. Then either

$$|A+A+A+A|^5|A:A|^4\gg{|A|^{10}},$$
or
$$|A+A+A+A|^4|A:A|^5\gg{|A|^{10}}.$$
In particular, it follows that
$$\max\{|A+A+A+A|,|A:A|\}\gg{|A|^{10/9}}.$$

\end{theorem}

\begin{proof} Simply repeat the proof of Theorem without applying Lemma \ref{theorem:CPlun} in the conclusion of each of the cases. In particular we obtain \eqref{case1.1} in case 1, \eqref{case2.1} in case 2, \eqref{case4.1} in case 4 and \eqref{case5.1} in case 5. Then, apply the trivial bound $|A+A+A+A|\geq{|A+A|}$ for each of these cases. Also, note that the conclusion \eqref{case3} in case 3 is already much stronger than the result claimed here.

\end{proof}

\section*{Acknowledgements}

I am grateful to Tim Jones, Liangpan Li and Igor Shparlinski for helpful conversations, and to Misha Rudnev for explaining how the proof of this estimate works in the prime fields case.

\end{document}